\documentclass[11pt]{article}

\usepackage[style=numeric,maxbibnames=99]{biblatex}
\usepackage{amsmath,amssymb,amsthm,color}
\usepackage{hyperref,url}
\usepackage{fancybox,graphicx} 
\usepackage[title]{appendix}

\newtheorem{theorem}{Theorem}[section]
\newtheorem{lemma}[theorem]{Lemma}

\newtheorem{proposition}[theorem]{Proposition}
\newtheorem{definition}[theorem]{Definition}
\newtheorem{assumption}[theorem]{Assumption}

\newtheorem{remark}[theorem]{Remark}

\newtheorem{example}{Example}

\newcommand{\XX}{{\cal X}}
\newcommand{\YY}{{\cal Y}}
\newcommand{\ZZ}{{\cal Z}}

\newcommand{\RR}{\mathbb R}

\newcommand{\Rmnum}[1]{\uppercase\expandafter{\romannumeral #1}} 
\usepackage[raggedright]{titlesec}
\titleformat{\chapter}{\centering\Huge\bfseries}{Chapter \Rmnum{\thechapter} }{1em}{} 

\usepackage{mathrsfs} 
\usepackage{enumerate} 
\graphicspath{{figures/}}

\newcommand{\VI}{\mbox{VI}}
\newcommand{\sol}{{\mbox{\rm Sol}}}

\newcommand{\proj}{{\mbox{\rm Proj}}}
\newcommand{\half}{0.5}
\newcommand{\T}{{\footnotesize\mbox{th}}}
\newcommand{\strongsol}{\sol(\VI(F;\XX))}
\newcommand{\weaksol}{\sol_m(\VI(F;\XX))}

\usepackage{algorithm}
\usepackage{algorithmicx}
\usepackage{multirow}
\usepackage{amsmath}
\usepackage{stmaryrd}
\usepackage{float}
\usepackage{graphicx}
\usepackage{subfigure}
\usepackage{xcolor}
\usepackage{algpseudocode}

\usepackage{color}

\begin{document}
\title{ Beyond Monotone Variational Inequalities:\\
Solution Methods and Iteration Complexities } 

\author{
Kevin Huang\thanks{Department of Industrial and System Engineering, University of Minnesota, huan1741@umn.edu}
\hspace{1cm}
Shuzhong Zhang\thanks{Department of Industrial and System Engineering, University of Minnesota, zhangs@umn.edu}
}

\date{\today}
\maketitle

\begin{abstract}
    In this paper, we discuss variational inequality (VI) problems without monotonicity from the perspective of convergence of projection-type algorithms. In particular, we identify existing conditions as well as present new conditions that are sufficient to guarantee convergence. The first half of the paper focuses on the case where a Minty solution exists (also known as {\it Minty condition}), which is a common assumption in the recent developments for non-monotone VI. The second half explores alternative sufficient conditions that are different from the existing ones such as monotonicity or Minty condition, using an {\it algorithm-based} approach. Through examples and convergence analysis, we show that these conditions are capable of characterizing different classes of VI problems where the algorithms are guaranteed to converge.
    
    \vspace{3mm}

\end{abstract}

\section{Introduction}
Let $\XX \subseteq \RR^n$ be a convex set, and $F(x): \XX \mapsto \RR^n$ be a vector mapping. The following problem is known as the {\it variational inequality problem}\/ (VI):
\[
\mbox{ Find $x^* \in \XX$ such that $F(x^*)^\top (x-x^*) \ge 0$ for all $x\in \XX$}.
\] 
The study of finite-dimensional VI problems dates back to 1960's where the complementarity problem was developed to solve for various equilibria, such as economic equilibrium, traffic equilibrium, and in general Nash equilibrium. For a comprehensive study of the applications, theories and algorithms of VI, readers are referred to the celebrated monograph by Facchinei and Pang \cite{facchinei2007finite}.

We say a problem is a {\it monotone} VI if the operator $F$ is monotone:
\begin{eqnarray}
\langle F(x)-F(y),x-y\rangle\ge\mu\|x-y\|^2,\quad\forall x,y\in\mathcal{X}\label{strong-monotone}
\end{eqnarray}
for some $\mu\ge0$. If there exists some $\mu>0$ such that \eqref{strong-monotone} holds, then $F$ is referred to as {\it strongly monotone}. 
Most contemporary researches on VI have focused on designing algorithms for monotone problems, such as modified forward-backward method \cite{tseng2000modified}, mirror-prox method \cite{nemirovski2004prox}, dual-extrapolation method \cite{nesterov2007dual, nesterov2006solving}, hybrid proximal extra-gradient method \cite{monteiro2010complexity}, OGDA~\cite{mokhtari2019unified, mokhtari2020convergence}, and extra-point method \cite{huang2021unifying}. These are also known as ``projection-type'' methods, while the earliest projection-type methods date back to (gradient) projection method by Sibony~\cite{sibony1970methodes}, proximal method by Martinet~\cite{martinet1970breve}, and extra-gradient method by Korpelevich \cite{korpelevich1976extragradient}. The convergence of these methods for monotone VI is studied by Tseng~\cite{tseng1995linear}. The aforementioned methods are all {\it first-order} projection methods. Recently there are also research on developing {\it high-order} projection methods and establishing their global convergence iteration complexities. They include~\cite{huang2020cubic, monteiro2012iteration, bullins2020higher, ostroukhov2020tensor, nesterov2006cubicVI, jiang2022generalized, adil2022optimal, lin2022perseus, huang2022approximation}.

The convergence of the above first-order and high-order (projection-type) methods and the corresponding iteration complexities are analyzed in the framework of monotone VI (or more generally, pseudo-monotone VI, which will be formally defined in the next section). For non-monotone VI, earlier research has developed non-projection-type methods such as the KKT condition based methods and merit function based methods (see \cite{facchinei1998regularity, qi1997semismooth, peng1999hybridgeneralVI, facchinei2007finite} and the references therein). However, it is in general difficult to establish iteration complexity for non-projection-type methods for non-monotone VI. In recent years, research on developing algorithms for non-monotone VI has focused on the VI problems where the so-called {\it Minty solutions} exist. A Minty solution to VI is a solution $x^*$ where the following inequality is satisfied:
\[
\langle F(x),x-x^*\rangle\ge0
\]
for all $x\in\XX$. When the constraint set is a close convex set and the operator $F$ is continuous and monotone, all solutions to the VI (if any) are Minty solutions. The existence of Minty solutions turns out to be critical in establishing convergence for the projection-type methods for non-monotone VI, and there have been recent developments of such results; see e.g.~\cite{lin2022perseus, song2020optimistic, burachik2020projection, ye2022infeasible, lei2021extragradient}.

In this paper, we follow this line of research on the convergence of projection-type methods for non-monotone VI. We start from the common assumption made in the literature, that is, a Minty solution exists. We show that a general extra-gradient-type method, the ARE update proposed in~\cite{huang2022approximation}, converges in a guaranteed rate with a similar convergence behavior as Perseus in~\cite{lin2022perseus}. In addition, we are interested in the concept of Minty solution itself, especially in the non-monotone setting where a VI solution is not necessarily a Minty solution. Therefore, we investigate implications given by the Minty solutions in different problem classes such as optimization and Nash games. Finally, we explore the possibilities of alternative sufficient conditions for convergence of projection-type methods through an {\it algorithm-based} approach. Conventionally, algorithms are devised to ensure convergence under a given problem framework, such as monotone VI or VI with Minty solutions, and the convergence behavior is analyzed within the framework. In this paper, we follow an opposite direction by deriving sufficient conditions for convergence based on the algorithms we are interested in. In other words, for a given algorithm, we aim to identify VI with specific structures where the algorithm is guaranteed to converge to a solution. It turns out that this approach makes it possible to characterize structures of VI models that are different from commonly encountered conditions such as the monotonicity or the Minty condition. We present several conditions of this kind and demonstrate examples as well as proving convergence of gradient projection methods and extra-gradient methods under these conditions.

The rest of the paper is organized as follows. Section~\ref{sec:non-mono-Minty} starts the discussion with non-monotone VI with Minty solution. We first provide formal definitions of the solution concepts and merit functions that are relevant to the discussion in this paper. Then the convergence of ARE is established, followed by discussions on implications of Minty solutions in optimization and Nash games. Section~\ref{sec:non-mon-no-minty} explores algorithm-based sufficient conditions, starting with formal definitions and examples. Convergence of gradient-projection method and extra-gradient method are established under these conditions. Finally, Section~\ref{sec:conclusion} concludes the paper 
with some further remarks.

\section{Non-monotone VI with Minty Solution}
\label{sec:non-mono-Minty}

\subsection{Definitions and solution concepts}

\subsubsection{VI solutions and Minty solutions}\label{sec:sol-VI-Minty}

In order to set the background for the discussion in this paper, let us first formally define the variational inequality problem and its solution set. For a given set $\XX\subseteq\RR^n$ and a continuous mapping $F:\: \XX\mapsto \RR^n$, consider the following VI model, to be denoted by $\VI(F;\XX)$:
\[
\begin{array}{ll}
\mbox{Find} & x^* \in \XX \\
\mbox{such that} & \langle F(x^*), x-x^*\rangle \ge 0,\,\, \forall x \in \XX.
\end{array}
\]
Let the solution set of the above model be $\sol(\VI(F;\XX))$. It is also referred to as the set of \textit{strong solutions}, or simply the \textit{solutions}, to the VI model. The non-emptiness of $\sol(\VI(F;\XX))$ can be guaranteed by imposing some assumptions on the basic problem structure.

\begin{assumption}\label{ass:problem-structure}
$F$ is a continuous mapping. $\XX$ is non-empty, convex and compact.
\end{assumption}

Assumption~\ref{ass:problem-structure} ensures that $\sol(\VI(F;\XX))\neq \emptyset$ \cite{facchinei2007finite, harker1990finite}, and we shall make this assumption throughout the paper. In addition to the (strong) solutions to the VI, there is another important solution concept, which is the so-called \textit{Minty solutions} or \textit{weak solutions}, defined as the set of $x^*$ such that
\[
 \langle F(x), x-x^*\rangle \ge 0,\,\, \forall x \in \XX.
\]
Let the set of Minty solutions to $\VI(F;\XX)$ be denoted by $\sol_m(\VI(F;\XX))$. The well-known Minty's Lemma states the follwing:
\begin{lemma}[Minty's Lemma]
    \label{lem:Minty}
    If $F$ is continuous, $\XX$ is non-empty, closed and convex, then $\weaksol\subseteq \strongsol$.
\end{lemma}
If additionally $F$ is monotone, then $\weaksol = \strongsol$. Indeed, for every $x^*\in\strongsol$, we have:
\[
\langle F(x),x-x^*\rangle\ge \langle F(x^*),x-x^*\rangle\ge0,
\]
thus $x^*\in\weaksol$. In this paper, while we always assume the non-emptiness of $\strongsol$ (by Assumption~\ref{ass:problem-structure}), we extend the discussion to the broader class of VI where $F$ is not necessarily monotone. Alternatively, we focus on the Minty solutions and assume $\weaksol\neq\emptyset$ in this section, and we discuss other conditions in Section~\ref{sec:non-mon-no-minty} where 
no Minty solutions exist.

\subsubsection{Merit functions}
\label{sec:merit-function}

In the context of strongly monotone VI where the solution $x^*$ uniquely exists, it is common to use (squared) distance to the solution $\|x-x^*\|^2$ in the iteration complexity analysis. For VI problems that are merely monotone, there are two other merit functions that are widely used, known as the {\it gap function} and the {\it dual gap function}. While monotonicity is not assumed in this paper, we may still use these two merit functions as the measurement of convergence. In this section we re-introduce them in a fashion that relates them with the two solution concepts, VI solutions $\strongsol$ and Minty solutions $\weaksol$, introduced earlier.

\begin{proposition}
Suppose that $\XX$ is compact. It holds that
\[
\sol(\VI(F;\XX)) \not= \emptyset \Longleftrightarrow \min_{y\in \XX} \max_{x\in \XX} \langle F(y),y-x\rangle =0.
\]
\end{proposition}

\begin{proof}
Observe that for any $y\in \XX$, we always have
\[
\max_{x\in \XX} \langle F(y),y-x\rangle \ge \langle F(y),y-y\rangle = 0.
\]
Hence, $\min_{y\in \XX} \max_{x\in \XX} \langle F(y),y-x\rangle \ge 0$, or $\max_{y\in \XX} \min_{x\in \XX} \langle F(y),x-y\rangle \le 0$.

$\Longrightarrow:$ Choose any $x^* \in \sol(\VI(F;\XX))$. We have $\min_{x\in \XX} \langle F(x^*),x-x^*\rangle =0$, implying
\[
\max_{y\in \XX} \min_{x\in \XX} \langle F(y),x-y\rangle = 0.
\]

$\Longleftarrow:$ Let
\[
y^* \in \mbox{\rm arg} \max_{y\in \XX} \left[\min_{x\in \XX} \langle F(y),x-y\rangle\right].
\]
It follows that $\min_{x\in \XX} \langle F(y^*),x-y^*\rangle=0$, or equivalently put
\[
\langle F(y^*),x-y^*\rangle \ge 0, \mbox{ for any } x \in \XX.
\]
Hence, $y^* \in \sol(\VI(F;\XX))$.
\end{proof}

In a similar vein, we have:
\begin{proposition}
Suppose that $\XX$ is compact. It holds that
\[
\sol_m(\VI(F;\XX)) \not= \emptyset \Longleftrightarrow \max_{x\in \XX} \min_{y\in \XX} \langle F(y),y-x\rangle =0.
\]
\end{proposition}

\begin{proof}
First, observe that in general, we have
\[
\max_{x\in \XX} \min_{y\in \XX} \langle F(y),y-x\rangle \le 0
\]
because for any given $x\in \XX$, it follows that $\min_{y\in \XX} \langle F(y),y-x\rangle \le \langle F(x),x-x\rangle = 0$.

$\Longrightarrow:$ Choose any $x^* \in \sol_m(\VI(F;\XX))$. Since $\langle F(y),y-x^*\rangle \ge 0$ for all $y \in \XX$, we have
\[
\min_{y\in \XX} \langle F(y),y-x^*\rangle = 0,
\]
implying $\max_{x\in \XX} \min_{y\in \XX} \langle F(y),y-x\rangle =0$.

$\Longleftarrow:$ Let
\[
x^* \in \mbox{\rm arg} \max_{x\in \XX} \left[\min_{y\in \XX} \langle F(y),y-x\rangle \right].
\]
We have $\min_{y\in \XX} \langle F(y),y-x^*\rangle =0$, or equivalently
\[
\langle F(y),y - x^*\rangle \ge 0, \mbox{ for all } y \in \XX,
\]
implying $x^* \in \sol_m(\VI(F;\XX))$.
\end{proof}

The above analysis naturally leads to the following notions of merit functions:
\[
G(x):=\max_{y\in \XX} \langle F(x),x-y\rangle,
\]
also known as the {\it gap function}, and
\[
H(x) : = \max_{y\in \XX} \langle F(y),x-y\rangle, 
\]
also known as the {\it dual gap function}.

Based on our analysis, we have:

\begin{proposition} {\ }

\begin{itemize}
\item $G(x) \ge 0$ for all $x\in \XX$, and $G(x)=0$ if and only if $x\in \sol(\VI(F;\XX))$.
\item $H(x) \ge 0$ for all $x\in \XX$, and $H(x)=0$ if and only if $x\in \sol_m(\VI(F;\XX))$.
\end{itemize}
\end{proposition}

Therefore, we may introduce the following notion of $\epsilon$-solutions.

\begin{definition}\label{def:epsilon-sol}
For $\epsilon>0$, we call $x$ to be an $\epsilon$-VI solution if $G(x) \le \epsilon$; we call $x$ to be an $\epsilon$-Minty solution if $H(x) \le \epsilon$.
\end{definition}

\subsubsection{Relaxation of monotonicity}
\label{sec:relax-monotone}

We remark that there are several conditions can be made on the structure of $F$, under which the (pure) monotonicity is relaxed but the connections between $\strongsol$ and $\weaksol$ 
still exist. 
While we do not assume most of these conditions, we summarize them below for 
the benefit of easy referencing. 
In particular, we only consider the {\it Minty condition}\/ among others, which simply states $\weaksol\neq\emptyset$.

\begin{itemize}
    \item Weak Sharpness:
    \begin{eqnarray}
    \langle F(x^*),x-x^*\rangle\ge \mu\|x-x^*\|^2,\quad\forall x\in\XX,\quad x^*\in\sol(\VI(F,\XX)),\nonumber
    \end{eqnarray}
    for some $\mu\ge0$.
    \item Pseudo-monotonicity:
    \begin{eqnarray}
    \langle F(y),x-y\rangle\ge0\Longrightarrow \langle F(x),x-y\rangle\ge0,\quad\forall x,y\in\XX.\nonumber
    \end{eqnarray}
    \item Strongly pseudo-monotonicity:
    \begin{eqnarray}
    \langle F(y),x-y\rangle\ge0\Longrightarrow \langle F(x),x-y\rangle\ge\mu\|x-y\|^2,\quad\forall x,y\in\XX,\nonumber
    \end{eqnarray}
    for some $\mu>0$.
    \item Quasi-monotonicity:
    \begin{eqnarray}
    \langle F(y),x-y\rangle>0\Longrightarrow \langle F(x),x-y\rangle\ge0,\quad\forall x,y\in\XX.\nonumber
    \end{eqnarray}
    \item Minty's condition: $\weaksol\neq\emptyset$, i.e.\ there exists $x^*\in\XX$ such that
    \begin{eqnarray}
    \langle F(x),x-x^*\rangle\ge0,\quad\forall x\in\XX.\nonumber
    \end{eqnarray}
    \item Strong Minty's condition (Generalized monotonicity): there exists $x^*\in\XX$ such that
    \begin{eqnarray}
    \langle F(x),x-x^*\rangle\ge\mu\|x-x^*\|^2,\quad\forall x\in\XX,\nonumber
    \end{eqnarray}
    for some $\mu\ge0$.
\end{itemize}

A few remarks are in place to specify some implications given by the above conditions.

\begin{remark}{\ }

\begin{itemize}
    \item If $F$ is pseudo-monotone, then $\strongsol\subset\weaksol$. If further $F$ is continuous, $\XX$ is nonempty, closed and convex, then $\strongsol=\weaksol$.
    \item If $\XX$ is closed and bounded, then $\weaksol\neq\emptyset$ if and only if $F$ is quasi-monotone \cite{he2017solvability}.
    \item Assume $\strongsol\neq\emptyset$, then the following relations hold:
\[
\mbox{monotone $\Longrightarrow$ pseudo-monotone $\Longrightarrow$ Minty's condition}
\]
and
\[
\mbox{strongly monotone $\Longrightarrow$ strongly pseudo-monotone $\Longrightarrow$ strong Minty's condition}
\]
\end{itemize}

\end{remark}



\subsection{Convergence of projection-type methods}
\label{sec:are-conv}

In this section, we present a solution method of projection type that can be shown to converge to Minty solutions by simply assuming $\weaksol\neq\emptyset$. This method, known as {\it Approximation-based Regularized Extra-gradient method}, or simply ARE, was a general $p^\T$-order method of extra-gradient type, proposed in \cite{huang2022approximation} originally for solving monotone VI with convergence rate $\mathcal{O}(N^{-\frac{p+1}{2}})$. It turns out that ARE not only solves monotone VI, but also solves non-monotone VI as long as $\weaksol\neq\emptyset$. A similar technique in the analysis has been used in \cite{lin2022perseus} to show a different projection-type method---Perseus---can converge to Minty solutions for non-monotone VI at the same rate as to be developed later in this paper. 

The ARE update proceeds as follows:
\begin{eqnarray}
\left\{
\begin{array}{ccl}
x^{k+0.5} &:=& \VI_{\XX} \left(  \tilde F(x;x^k) + L_p \|x-x^k\|^{p-1} (x-x^k) \right) , \\
x^{k+1} &:=& \mbox{arg}\min\limits_{x \in \XX} \, \langle F(x^{k+0.5}) , x - x^k \rangle + \frac{L_p\|x^{k+0.5}-x^k\|^{p-1}}{2} \|x-x^k\|^2,
\end{array}
\right.\label{are-update}
\end{eqnarray}
for $k=1,2,...$, where $L_p$ is the Lipschitz constant for $\nabla^{p-1}F(x)$ satisfying the condition
\begin{eqnarray}
\|\nabla^{p-1}F(x)-\nabla^{p-1}F(y)\|\le L_p\|x-y\|,\label{p-lipschitz}
\end{eqnarray}
$\tilde F(\cdot;y):\RR^n\mapsto\RR^n$ is a general approximation mapping estimated at $y$ satisfying the bound:
\begin{eqnarray}
\|\Tilde{F}(x;y)-F(x)\|\le \tau L_p\|x-y\|^p,\label{approx-oper-bound}
\end{eqnarray}
and we use the notation $\VI_\XX(F)$ to denote solving for a solution in $\strongsol$ as a subroutine in the update. In the ARE update~\eqref{are-update}, the subroutine at iteration $k$ specifically solve the VI model associated with the regularized approximation operator $\tilde F(x;x^k) + L_p \|x-x^k\|^{p-1} (x-x^k)$.

The initial steps in the analysis will follow the exact same logic as the proof in~\cite{huang2022approximation} for Theorem 3.1. By the definition of $x^{k+\half}$, we have
\begin{equation} \label{subVI}
\langle \tilde F(x^{k+\half};x^k) + L_p \|x^{k+\half}-x^k\|^{p-1} (x^{k+\half}-x^k) , x - x^{k+\half}\rangle \ge 0,\,\, \forall x \in \XX.
\end{equation}
Denote $\gamma_k=L_p \|x^{k+\half}-x^k\|^{p-1}$.
Substituting $x=x^{k+1}$ in \eqref{subVI} we have
\begin{eqnarray}
&&\langle \tilde F(x^{k+\half};x^k)  ,  x^{k+1} - x^{k+\half}  \rangle \nonumber\\
&\ge& \gamma_k \langle x^{k+\half} - x^k , x^{k+\half} - x^{k+1}  \rangle \nonumber \\
&=& \frac{\gamma_k}{2} \left( \|x^{k+\half} - x^k\|^2 + \|x^{k+1}-x^{k+\half}\|^2 - \| x^{k+1}-x^k\|^2 \right) . \label{tildeF}
\end{eqnarray}

On the other hand, by the optimality condition at $x^{k+1}$ we have
\[
\langle F(x^{k+\half})  + \gamma_k (x^{k+1} - x^k) , x - x^{k+1} \rangle \ge 0 ,\mbox{ for all $x\in \XX$}.
\]
Hence,
\begin{eqnarray}
\langle F(x^{k+\half})  , x - x^{k+1} \rangle &\ge& \gamma_k \langle x^{k+1} - x^k, x^{k+1} - x \rangle  \nonumber \\
&=& \frac{\gamma_k}{2} \left(  \| x^{k+1} - x\|^2 + \| x^{k+1} - x^k\|^2 - \| x^k - x\|^2 \right) ,\, \mbox{ for all $x\in \XX$}. \label{F ineq}
\end{eqnarray}

Continue with the above inequality, for any given $x\in \XX$ we have
\begin{eqnarray*}
& & \frac{\gamma_k}{2} \left(  \| x^{k+1} - x\|^2 + \| x^{k+1} - x^k\|^2 - \| x^k - x\|^2 \right) \\
&\overset{\eqref{F ineq}}{\le} & \langle F(x^{k+\half})  , x - x^{k+1} \rangle \\
&=&  \langle F(x^{k+\half})  , x - x^{k+\half} \rangle + \langle F(x^{k+\half})  , x^{k+\half} - x^{k+1} \rangle \\
&=& \langle F(x^{k+\half})  , x - x^{k+\half} \rangle + \langle F(x^{k+\half}) - \tilde F(x^{k+\half};x^k) , x^{k+\half} - x^{k+1} \rangle + \langle \tilde F(x^{k+\half};x^k)  , x^{k+\half} - x^{k+1} \rangle \\
&\le& \langle F(x^{k+\half})  , x - x^{k+\half} \rangle + \| F(x^{k+\half}) - \tilde F(x^{k+\half};x^k) \| \cdot \| x^{k+\half} - x^{k+1} \|
+ \langle \tilde F(x^{k+\half};x^k)  , x^{k+\half} - x^{k+1} \rangle \\
&\le & \langle F(x^{k+\half})  , x - x^{k+\half} \rangle + \frac{\| F(x^{k+\half}) - \tilde F(x^{k+\half};x^k)\|^2}{2\gamma_k}  + \frac{\gamma_k \|  x^{k+\half} - x^{k+1} \|^2}{2}\\
& & + \langle \tilde F(x^{k+\half};x^k)  , x^{k+\half} - x^{k+1} \rangle \\
&\overset{\eqref{approx-oper-bound}}{\le} & \langle F(x^{k+\half})  , x - x^{k+\half} \rangle + \frac{\tau^2 L_p^2 \| x^{k+\half} - x^k\|^{2p} }{2\gamma_k}  + \frac{\gamma_k \|  x^{k+\half} - x^{k+1} \|^2}{2}\\
& & + \langle \tilde F(x^{k+\half};x^k)  , x^{k+\half} - x^{k+1} \rangle .
\end{eqnarray*}
Noticing that $\frac{\tau^2 L_p^2 \| x^{k+\half} - x^k\|^{2p}}{2\gamma_k}=\frac{\tau^2 \gamma_k \| x^{k+\half} - x^{k} \|^2}{2}$, and further using \eqref{tildeF} we derive from the above that
\begin{eqnarray*}
& & \frac{\gamma_k}{2} \left(  \| x^{k+1} - x\|^2 + \| x^{k+1} - x^k\|^2 - \| x^k - x\|^2 \right) \\
&\le& \langle F(x^{k+\half})  , x - x^{k+\half} \rangle + \frac{\tau^2 \gamma_k \| x^{k+\half} - x^{k} \|^2}{2} + \frac{\gamma_k \|  x^{k+\half} - x^{k+1} \|^2}{2} \\
&   & + \frac{\gamma_k}{2} \left[ - \|x^{k+\half} - x^k\|^2 - \|x^{k+1}-x^{k+\half}\|^2 + \| x^{k+1}-x^k\|^2 \right] .
\end{eqnarray*}
Canceling out terms, we simplify the above inequality into
\begin{equation} \label{iter}
\langle F(x^{k+\half})  , x^{k+\half} - x \rangle + \frac{\gamma_k}{2} \left(1-\tau^2 \right) \| x^{k+\half} - x^{k} \|^2 \le
\frac{\gamma_k}{2} \left[ \|x^{k} - x\|^2 - \|x^{k+1}-x\|^2  \right] .
\end{equation}

In the original analysis in~\cite{huang2022approximation}, the rest of the proof continues with the monotonicity of $F$. In this analysis, we assume the Minty condition (i.e.\ $\weaksol\neq\emptyset$) holds instead of monotonicity of $F$. Taking any fixed $x=x^*\in\weaksol$ in the above inequality, we have:
\begin{eqnarray}\label{conv-to-Minty}
\left(1-\tau^2 \right) \| x^{k+0.5} - x^{k} \|^2 \le \left[ \|x^{k} - x^*\|^2 - \|x^{k+1}-x^*\|^2  \right],
\end{eqnarray}
since $\langle F(x^{k+0.5}),x^{k+0.5}-x^*\rangle\ge0$. Summing up this inequality for $k=1,...,N$ gives us:
\begin{eqnarray}
\min\limits_{1\le k\le N}\|x^{k+0.5}-x^k\|^2\le \frac{1}{N}\sum\limits_{k=1}^N\|x^{k+0.5}-x^k\|^2\le\frac{1}{N(1-\tau^2)}\|x^1-x^*\|^2.\label{conv-of-iterate}
\end{eqnarray}

Condition~\eqref{subVI} for updating $x^{k+0.5}$ implies for all $x\in\XX$, we have:
\begin{eqnarray}
\langle\tilde F(x^{k+0.5};x^k),x^{k+0.5}-x\rangle&\le& -L_p\|x^{k+0.5}-x^k\|^{p-1}(x^{k+0.5}-x^k)^\top(x^{k+0.5}-x)\nonumber\\
&\le& L_pD\|x^{k+0.5}-x^k\|^p,\nonumber
\end{eqnarray}
where $D:=\max\limits_{x,x'\in\mathcal{X}}\|x-x'\|$.

Denote $k_N:=\arg\min\limits_{1\le k \le N}\|x^{k+0.5}-x^k\|^2$, we have:
\begin{eqnarray}
&&\langle F(x^{k_N+0.5}),x^{k_N+0.5}-x\rangle\nonumber\\
&=& \langle F(x^{k_N+0.5})-\tilde{F}(x^{k_N+0.5};x^{k_N}),x^{k_N+0.5}-x\rangle+\langle \tilde F(x^{k_N+0.5};x^{k_N}),x^{k_N+0.5}-x\rangle\nonumber\\
&\le& \left\|F(x^{k_N+0.5})-\tilde{F}(x^{k_N+0.5};x^{k_N})\right\|\cdot\|x^{k_N+0.5}-x\|+ L_pD\|x^{k_N+0.5}-x^{k_N}\|^p\nonumber\\
&\le& (1+\tau) L_p D\|x^{k_N+0.5}-x^{k_N}\|^p\nonumber\\
&\le& (1+\tau) L_p D\frac{1}{N^{\frac{p}{2}}(1-\tau^2)^\frac{p}{2}}\|x^1-x^*\|^p\le \frac{(1+\tau)L_p D^{p+1}}{N^{\frac{p}{2}}(1-\tau^2)^\frac{p}{2}},\nonumber
\end{eqnarray}
which holds for all $x\in\XX$. Therefore,
\begin{eqnarray}
G(x^{k_N+0.5})=\max\limits_{x\in\XX}\,\,\langle F(x^{k_N+0.5}),x^{k_N+0.5}-x\rangle \le \frac{(1+\tau)L_p D^{p+1}}{N^{\frac{p}{2}}(1-\tau^2)^\frac{p}{2}}.\nonumber
\end{eqnarray}

We summarize the above results in the next theorem.
\begin{theorem}\label{thm:are-Minty}
    Consider the ARE update~\eqref{are-update}. Suppose that conditions~\eqref{p-lipschitz} and~\eqref{approx-oper-bound} are satisfied, and $\weaksol\neq\emptyset$. Then, the sequence produced by ARE converges at the following rate:
    \[
    \|x^{k_N+0.5}-x^{k_N}\|^2=\mathcal{O}(1/N),\quad G(x^{k_N+0.5})=\mathcal{O}(1/N^{\frac{p}{2}}).
    \]
\end{theorem}

\begin{remark}
In the above analysis, by using the same $x^*\in\weaksol$ in~\eqref{conv-to-Minty} repetitively, it can be seen that the sequence $\{x^k\}$ converges to the specific Minty solution in terms of squared distance. On the other hand, while $\|x^{k_N+0.5}-x^{k_N}\|^2$ also converges at a rate $1/N$ by~\eqref{conv-of-iterate}, the final result guarantees a rate of convergence $1/N^{\frac{p}{2}}$ in terms of the merit function $G(x^{k_N+0.5})$, which gives an $\epsilon$-VI solution (but not necessarily an $\epsilon$-Minty solution) based on Definition~\ref{def:epsilon-sol}.  
\end{remark}

Finally, we note that if $F$ is monotone (in which case $\strongsol=\weaksol$), the convergence rate is $1/N^{\frac{p+1}{2}}$ in terms of the merit function $H(\bar x_N)$, where $\bar x_N$ is the weighted average of $x^{k+0.5}$ (see~\cite{huang2022approximation},  Theorem 3.1). It is an improved rate compared to the above result where we only assume $\weaksol\neq\emptyset$.

\subsection{Minty solutions beyond general VI}

In the previous section, we see that the existence of Minty solutions is indeed an important property to have for solving non-monotone VI. Without assuming any other conditions, it allows projection-type methods such as extra-gradient method (a first-order specialized method of ARE) to converge (to a Minty solution) with guaranteed rate. It is then natural to ask whether the same solution concept presents with similar significance in other problem classes related to VI and what are the implications of the Minty solution therein. In this section, we first discuss the role of the Minty solution in optimization. The discussion proceeds in the context of Nash games, where we present the implications of the Minty solution and its connections to the VI model.

\subsubsection{Minty solutions in optimization}

Consider the optimization problem:
\begin{eqnarray}
\min\limits_{x\in\XX}\,\, f(x),\label{opt-problem}
\end{eqnarray}
where $f(x)$ is continuously differentiable, $\XX$ is convex and closed. The local first-order optimality condition is given by:
\begin{equation}\label{first-order-optimization}
    \langle \nabla f(x^*),x-x^*\rangle\ge0,\quad\forall x\in\XX,
\end{equation}
which is equivalent to the VI model $\VI(\nabla f;\XX)$ with solution set $\sol(\VI(\nabla f;\XX))$. Now suppose a Minty solution exists for this VI model, that is, $\sol_m(\VI(\nabla f;\XX))\neq\emptyset$ and any element $x^*\in\sol_m(\VI(\nabla f;\XX))$ satisfies
\begin{equation}\label{Minty-optimization}
   \langle \nabla f(x),x-x^*\rangle\ge0,\quad\forall x\in\XX. 
\end{equation}
Note that Minty's Lemma applies here due to our assumption. Therefore, $\sol_m(\VI(\nabla f;\XX))\subseteq \sol(\VI(\nabla f;\XX))$, and any $x^*$ satisfying~\eqref{Minty-optimization} is a local first-order stationary point (\eqref{first-order-optimization} holds). The next Theorem states that a Minty solution in optimization, if exists, is in fact a global solution to the problem.

\begin{theorem}[Optimality of Minty solution]
    For the optimization problem~\eqref{opt-problem} where $f$ is continuously differentiable, $\XX$ is convex and closed. The following holds:
    \begin{eqnarray}
        &&\langle \nabla f(x),x-x^*\rangle\ge0,\quad\forall x\in\XX\nonumber\\
        &\Longrightarrow& f(x^*)\le f(x),\quad\forall x\in\XX.\nonumber
    \end{eqnarray}
    In other words, if a Minty solution exists, it is a global solution to the problem.
    \begin{proof}
        Using the following identity:
        \begin{eqnarray}
        f(x^*)-f(x)=\int\limits_{t=0}^1\nabla f(x+t(x^*-x))^\top(x^*-x)dt,\nonumber
        \end{eqnarray}
        we have for $t=1$, $\nabla f(x^*)^\top(x^*-x)\le0$ since~\eqref{first-order-optimization} holds. For $t=0$, we have $\nabla f(x)^\top(x^*-x)\le0$ due to~\eqref{Minty-optimization}. For $0<t<1$, let $\hat x=x+t(x^*-x)\in\XX$, then
        \begin{eqnarray}
        \nabla f(x+t(x^*-x))^\top(x^*-x)=\frac{1}{1-t}\nabla f(\hat x)^\top(x^*-\hat x)\le0,\quad 0<t<1,\nonumber
        \end{eqnarray}
        where the last inequality is again due to~\eqref{Minty-optimization}. Therefore, we can conclude that
        \begin{eqnarray}
        &&f(x^*)-f(x)=\int\limits_{t=0}^1\nabla f(x+t(x^*-x))^\top(x^*-x)dt\le0\nonumber\\
        &\Longrightarrow& f(x^*)\le f(x).\nonumber
        \end{eqnarray}
    \end{proof}
\end{theorem}

\begin{remark}
The Minty solution is always a global solution in optimization, provided that $f(x)$ is continuously differentiable and $\XX$ is closed convex set. However, a global solution needs not be a Minty solution. For a meaningful optimization problem where a global solution exists, a Minty solution may not exist.
\end{remark}

Consider a one-dimensional optimization problem:
\[
\min\limits_{-1\le x\le 1}\,\, -x^2,
\]
the global solutions are $x^*={-1,1}$. For $x^*=-1$:
\[
\langle \nabla f(x),x-x^*\rangle=\langle-2x,x+1\rangle<0,\,\,0<x\le1;
\]
for $x^*=1$,
\[
\langle \nabla f(x),x-x^*\rangle=\langle-2x,x-1\rangle<0,\,\,-1\le x<0.
\]
Therefore, neither of the global solutions is a Minty solution. Same as the VI model, when the objective function is convex, the set of local solutions ($\sol(\VI(\nabla f;\XX))$) coincides with the set of Minty solutions ($\sol_m(\VI(\nabla f;\XX))$), thus every global solution is a Minty solution.

\subsubsection{Minty solutions in Games}

Consider a two-player game:
\begin{eqnarray}
\left\{
\begin{array}{lcl}
     x:&& \min\limits_{x\in\XX}\,\, \theta_x(x,y)\\
    y:&& \min\limits_{y\in\YY}\,\, \theta_y(x,y),
\end{array}
\right.\label{bimatrix-game}
\end{eqnarray}
where we use $x,y$ to denote both the players and their corresponding strategies. Assume $\theta_x,\theta_y$ are both continuously differentiable for fixed $y,x$, and $\XX,\YY$ are closed convex sets. Let us first define three different notions of equilibria in this game, starting from the well-known Nash equilibrium.

\begin{definition}[Nash equilibrium (NE)]\label{def:Nash-eq}
    A solution pair $(x^*,y^*)\in\XX\times\YY$ is said to be in Nash equilibrium if and only if
    \[    \theta_x(x^*,y^*)\le\theta_x(x,y^*),\,\forall x\in\XX,\quad \theta_y(x^*,y^*)\le\theta_y(x^*,y),\,\forall y\in\YY.
    \]
\end{definition}
In other words, for player $x$ it is not possible to be better off by deviating from the Nash equilibrium strategy $x^*$ if the opponent continues to play $y=y^*$ and vice versa. For a Nash equilibrium pair $(x^*,y^*)$, $x^*/y^*$ is the global minimizer of the objective function $\theta_x(\cdot,y^*)/\theta_y(x^*,\cdot)$ for fixed $y^*/x^*$. 

\begin{definition}[Quasi-Nash equilibrium~\cite{pang2011nonconvex} (QNE)]\label{def:quasi-Nash-eq}
    A solution pair $(x^*,y^*)\in\XX\times\YY$ is said to be in quasi-Nash equilibrium if and only if
    \[
    \langle \nabla_x\theta_x(x^*,y^*),x-x^*\rangle\ge0,\,\forall x\in\XX,\quad \langle \nabla_y
    \theta_y(x^*,y^*),y-y^*\rangle\ge0,\,\forall y\in\YY.
    \]
\end{definition}
Unlike Nash equilibrium where $x^*$ and $y^*$ have to be global minimizers of their respective objective functions when the opponent plays the equilibrium strategy, a pair of quasi-Nash equilibrium only requires the first-order stationarity condition to be satisfied in their respective optimization problem. Hence quasi-Nash equilibrium can be viewed as a relaxation of Nash equilibrium.

\begin{definition}[Minty Nash equilibrium (MNE)]\label{def:minty-Nash-eq}
    A solution pair $(x^*,y^*)\in\XX\times\YY$ is said to be in Minty Nash equilibrium if and only if
    \[
    \langle \nabla_x\theta_x(x,y^*),x-x^*\rangle\ge0,\,\forall x\in\XX,\quad \langle \nabla_y
    \theta_y(x^*,y),y-y^*\rangle\ge0,\,\forall y\in\YY.
    \]
\end{definition}
The third definition given above pertains to the notion of Minty solution discussed thus far. It requires $x^*/y^*$ to be a Minty solution of $\theta_x(\cdot,y^*)/\theta_y(x^*,\cdot)$ for fixed $y^*/x^*$. By the discussion in the previous section, the set of Minty solutions is only a subset of the global solutions, therefore the Minty Nash equilibrium defines a stronger concept of equilibrium than the usual notion of Nash equilibrium.

It is straightforward to conclude the following relation among these three different notions of equilibria:
\begin{eqnarray}
\mbox{MNE}\Longrightarrow\mbox{NE}\Longrightarrow\mbox{QNE}.\nonumber
\end{eqnarray}
If the objective functions possess an additional property known as {\it block multiconvex}, i.e.\ $\theta_x(\cdot,y)$ is convex for fixed $y\in\YY$ and $\theta_y(x,\cdot)$ is convex for fixed $x\in\XX$, then the above relation becomes:
\begin{eqnarray}
\mbox{MNE}=\mbox{NE}=\mbox{QNE}.\nonumber
\end{eqnarray}

Let us now consider the connections among the above notions of equilibria to the solutions in the VI formulation of the two-player game~\eqref{bimatrix-game}:
\begin{eqnarray}
F(z):=\begin{pmatrix}
\nabla_x \theta_x(x,y)\\\nabla_y \theta_y(x,y)
\end{pmatrix},\quad z:=(x,y)^\top,\quad \ZZ:=\XX\times\YY,\label{game-VI-form}
\end{eqnarray}
which can be expressed as the VI model $\VI(F;\ZZ)$.

If $z^*=(x^*,y^*)\in\sol(\VI(F;\ZZ))$, i.e.\,
\begin{eqnarray}
\langle F(z^*),z-z^*\rangle=\langle\nabla_x\theta_x(x^*,y^*),x-x^*\rangle+\langle\nabla_y\theta_y(x^*,y^*),y-y^*\rangle\ge0,\quad\forall z\in\ZZ,\nonumber
\end{eqnarray}
it is obvious that $(x^*,y^*)$ is a pair of {quasi-Nash equilibrium} of the original two-player game by taking $x=x^*$ and $y=y^*$ in the above inequality. On the other hand, if $(x^*,y^*)$ is a pair of quasi-Nash equilibrium, then the above inequality holds trivially and $z^*=(x^*,y^*)\in\sol(\VI(F;\ZZ))$. Therefore, quasi-Nash equilibrium of the game is equivalent to the (strong) solution to the VI formulation.

Now if $z^*=(x^*,y^*)\in\sol_m(\VI(F;\ZZ))$, i.e.\,
\begin{eqnarray}
\langle F(z),z-z^*\rangle=\langle\nabla_x\theta_x(x,y),x-x^*\rangle+\langle\nabla_y\theta_y(x,y),y-y^*\rangle\ge0,\quad\forall z\in\ZZ.\label{Minty-ineq-game}
\end{eqnarray}
By taking any arbitrary $x\in\XX$ and $y=y^*$ in the above inequality, we have
\begin{eqnarray}
\langle\nabla_x\theta_x(x,y^*),x-x^*\rangle\ge0,\quad\forall x\in\XX.\nonumber
\end{eqnarray}
Similarly we have
\begin{eqnarray}
\langle\nabla_y\theta_y(x^*,y),y-y^*\rangle\ge0,\quad\forall y\in\YY.\nonumber
\end{eqnarray}
The above two inequalities combined indicates that $(x^*,y^*)$ is a pair of Minty Nash equilibrium of the original two-player game. However, we note that the opposite direction is in general not true, since a Minty solution $z^*$ to $\VI(F;\XX)$ requires the inequality~\eqref{Minty-ineq-game} to be satisfied with $\nabla_x\theta_x(x,y)/\nabla_y\theta_y(x,y)$ while the Minty Nash equilibrium only defines on $\nabla_x\theta_x(x,y^*)/\nabla_y\theta_y(x^*,y)$, in which the opponent's strategy is fixed to be the equilibrium strategy.

We can include the two solution concepts in the VI formulation, $\sol(\VI(F;\ZZ))$ and $\sol_m(\VI(F;\ZZ))$, in the previous relation and obtain
\begin{eqnarray}
\sol_m(\VI(F;\ZZ))\Longrightarrow\mbox{MNE}\Longrightarrow\mbox{NE}\Longrightarrow\mbox{QNE}=\sol(\VI(F;\ZZ)).\nonumber
\end{eqnarray}
In the case where both $\theta_x(x,y)$ and $\theta_y(x,y)$ are bock multiconvex, then 
\begin{eqnarray}
\sol_m(\VI(F;\ZZ))\Longrightarrow\mbox{MNE}=\mbox{NE}=\mbox{QNE}=\sol(\VI(F;\ZZ)).\nonumber
\end{eqnarray}
Note that it is not sufficient to state the equivalence between the Minty solution in VI and others even if the objective functions are block multiconvex. However, the above relation does offer a quick argument for the existence of Nash equilibrium for non-cooperative games where the payoff functions are block multiconvex and continuously differentiable, and the constraints are convex compact sets. Indeed, the latter two conditions are exactly given in Assumption~\ref{ass:problem-structure}, which guarantees that $\sol(\VI(F;\ZZ))\neq\emptyset$. The functions being block multiconvex indicates that NE$=\sol(\VI(F;\ZZ))$, which proves the existence of Nash equilibrium. This conclusion is summarized in the next proposition.

\begin{proposition}
    For the two-player game~\eqref{bimatrix-game}, if both $\theta_x(x,y)$ and $\theta_y(x,y)$ are block multiconvex and continuously differentiable, and $\XX,\YY$ are convex compact sets, then a Nash equilibrium exists.
\end{proposition}
\begin{remark}
    Indeed, the celebrated Nash Theorem has shown that a mixed strategy Nash equilibrium exists in $n$-player multilinear games using Brouwer's fixed-point theorem. The above discussion only points out an alternative route for showing the same conclusion with a potentially lighter algebraic derivation. The equivalence between NE and QNE as well as between QNE and $\sol(\VI(F;\ZZ))$ is straightforward. Proving $\sol(\VI(F;\ZZ))\neq\emptyset$ can be accomplished via the same Brouwer's fixed-point theorem or degree theory \cite{facchinei2007finite}, where the details are omitted here.
\end{remark}

We note that restricting the number of players to be two in this section is only for the purpose of clear illustrations of the ideas. All the discussions can be easily extended to general $n$-player games.

\section{Algorithm-Based Conditions on VI}
\label{sec:non-mon-no-minty}

In the discussion thus far, we have focused on non-monotone VI with the condition $\weaksol\neq\emptyset$, and we show that it is sufficient for a projection-type method such as ARE to converge globally with a guaranteed rate. In this section, we continue to ask the question: Are there other sufficient conditions different from the existing ones that are able to guarantee the convergence of algorithms of certain class? It turns out that through an {\it algorithm-based} approach, it is possible to characterize structures of the VI model by deriving sufficient conditions for which the algorithms converge. These VI problems can be of special interest since they are not necessarily monotone or satisfy the Minty condition, nonetheless the algorithms converge regardless. In particular, we present conditions on VI models based on projection-type methods, analyze their convergence behavior, and provide examples of problems satisfying these conditions.

\subsection{Conditions for projection-type methods}

In order to present the conditions to be introduced later with more precise expressions, let us first define two projection-type mappings, which play a central role in these conditions since the purpose is to characterize VI problems with guaranteed convergence for projection-type methods. 

\begin{definition}[Gradient projection mapping]\label{def:gp-map}
    For a given $t>0$, assuming $\XX$ is a closed convex set, define the ``gradient projection mapping'' as \begin{eqnarray}
    M(x;t):= \proj_{\XX}(x-t F(x)).\label{gp-map}
    \end{eqnarray}
\end{definition}
Note that the term ``gradient'' follows the convention in optimization, while in general $F$ can be any vector mapping that is not necessarily a gradient mapping.

It is a well-known fact that for a fixed $t>0$, $x^*\in\XX$ is a fixed point of the gradient projection mapping $M(\cdot;t)$ if and only if $x^*\in\strongsol$. It is then natural to use a third merit function other than the (dual) gap function introduced in Section~\ref{sec:merit-function}:
\[
P(x):= \| M(x;t)-x\|^2.
\]
We summarize the above observations in the next proposition and provide a proof for completeness.
\begin{proposition}
    $x^* = M(x^*;t)$ if and only if $x^*\in \sol(\VI(F;\XX))$. Therefore, $P(x)=0$ if and only if $x\in \sol(\VI(F;\XX))$.
    \begin{proof}
        By the optimality condition of the projection operation, we have
        \begin{equation} \label{projection optimaility}
        (y-M(x;t))^\top (M(x;t)- x + t F(x)) \ge 0 ,\,\,\, \forall y \in \XX.
        \end{equation}
        If $M(x;t)\not= x$, then by setting $y=x$ in \eqref{projection optimaility} we observe
        \[
        (M(x;t)-x)^\top F(x) \le - \frac{1}{t} \, \|M(x;t)-x\|^2 < 0
        \]
        implying that $x\not \in \sol(\VI(F;\XX))$. On the other hand, if $M(x;t)=x$, then \eqref{projection optimaility} yields
        \[
        t (y-x)^\top F(x) \ge 0 ,\,\,\, \forall y \in \XX,
        \]
        and so $x\in \sol(\VI(F;\XX))$.
    \end{proof}
\end{proposition}
\begin{remark}
    In the literature, the mapping $M(x;1)-x$ is also referred to as the ``natural map'' and solving $M(x;1)-x=0$ can be used as an equation reformulation of the VI model $\VI(F;\XX)$. In our discussion, we do not explicitly adopt the equation reformulation approach, but only use $P(x)$ as one of the measurements of convergence.
\end{remark}

In view of the gradient projection mapping defined earlier, let us also define the following ``extra-gradient projection mapping'', which expresses the extra-gradient-type methods and the sufficient condition for convergence more succinctly.
\begin{definition}[Extra-gradient projection mapping]\label{def:ex-gp-map}
    For a given $t>0$, assuming $\XX$ is a closed convex set, define the ``extra-gradient mapping'' as
    \begin{eqnarray}
    M^+(x;t):= \proj_{\XX}(x-t F(M(x;t))),\nonumber
    \end{eqnarray}
    where $M(x;t)$ is the gradient projection mapping defined in~\eqref{gp-map}.
\end{definition}

We are now ready to introduce conditions that can guarantee the convergence for different projection-type methods. These conditions provide additional characterizations of the structure of a VI problem when in general we do not assume monotonicity nor Minty condition. 

\begin{definition}\label{def:local-minty}
Condition (Local Minty).

For some fixed $t>0$ and any $x\in \XX$ there is $x^* \in \sol(\VI(F;\XX))$ such that
\[
\langle F(x),x-x^*\rangle \ge 0,
\]
and for the same $x^*$ the above inequality also holds if we replace $x$ by $M(x;t)$.
\end{definition}

\begin{definition}\label{def:local-minty-plus}
Condition (Local Minty+).

For some fixed $t>0$ and any $x\in \XX$ there is $x^* \in \sol(\VI(F;\XX))$ such that
\[
\langle F(M(x;t)),M(x;t)-x^*\rangle \ge 0,
\]
and for the same $x^*$ the above inequality also holds if we replace $x$ by $M^+(x;t)$.
\end{definition}

\begin{definition}\label{def:local-minty-star}
Condition (Local Minty*).

For some fixed $t>0$ and any $x\in \XX$ there is $x^* \in \sol(\VI(F;\XX))$ such that
\[
\langle F(x),M(x;t)-x^*\rangle \ge 0,
\]
and for the same $x^*$ the above inequality also holds if we replace $x$ by $M(x;t)$.
\end{definition}

\begin{definition}\label{def:gp}
Condition (GP).

For some fixed $\delta>0$ and $t>0$, and any $x\in \XX$ there is $x^* \in \sol(\VI(F;\XX))$ such that
\[
4(1+\delta)t \langle F(M(x;t)),M(x;t)-x^*\rangle + \| M(x;t)-x\|^2 \ge 0,
\]
and for the same $x^*$ the above inequality also holds if we replace $x$ by $M(x;t)$.
\end{definition}

\begin{definition}\label{def:gp-plus}
Condition (GP+).

For some fixed $\delta>0$ and $t>0$, and any $x\in \XX$ there is $x^* \in \sol(\VI(F;\XX))$ such that
\[
4(1+\delta)t \langle F(M(x;t)),M(x;t)-x^*\rangle + \| M(x;t)-x\|^2 \ge 0,
\]
and for the same $x^*$ the above inequality also holds if we replace $x$ by $M^+(x;t)$.
\end{definition}

\begin{definition}\label{def:gp-star}
Condition (GP*).

For some fixed $\delta>0$ and $t>0$, and any $x\in \XX$ there is $x^* \in \sol(\VI(F;\XX))$ such that
\[
2 (1+\delta) t \langle F(x),M(x;t)-x^*\rangle + \| M(x;t)-x\|^2 \ge 0 ,
\]
and for the same $x^*$ the above inequality also holds if we replace $x$ by $M(x;t)$.
\end{definition}

We make the following remarks on the conditions introduced above.
\begin{remark}{\ }

    \begin{itemize}
        \item Condition (Local Minty) (Local Minty*): Unlike other conditions that are seen before, both conditions are defined on sequences in $\XX$ rather than arbitrary points. In particular these sequences are generated by the gradient projection mapping $M(x;t)$. For any specific sequence, there is a ``local Minty solution'' for which the inequality defined in the respective condition continues to hold. 
        \item Condition (Local Minty+): When generating a sequence from the extra-gradient projection mapping $M^+(x;t)$, we immediately obtain another sequence that maps the previous sequence to their gradient projection mapping $M(x;t)$. (Local Minty+) is defined on the latter sequences.
        \item Condition (GP), (GP+), (GP*): These three conditions can be viewed as relaxations of Condition (Local Minty), (Local Minty+), and (Local Minty*), by allowing a positive term $P(x)$ in the defining inequality.
    \end{itemize}
\end{remark}

The conditions introduced above have one property in common: they all require the defining inequality to hold for the whole sequence generated from some pre-determined mapping. In other words, a solution $x^*\in\strongsol$ can ``attract'' some $x\in\XX$ following the particular sequence. Condition (Local Minty)/(Local Minty+)/(Local Minty*) (Definition~\ref{def:local-minty}, \ref{def:local-minty-plus}, and~\ref{def:local-minty-star}) guarantee the existence of such ``local Minty solution'' $x^*$ with respect to arbitrary $x\in\XX$, while Condition (GP)/(GP+)/(GP*) (Definition~\ref{def:gp}, \ref{def:gp-plus}, and~\ref{def:gp-star}) relax the previous three conditions. The term ``local'' is in contrast to the normal Minty solution, which is ``global'' since any $x^*\in\weaksol$ is able to attract {\it every} point in $\XX$ in terms of the Minty inequality $\langle F(x),x-x^*\rangle\ge0$ in the definition. This property turns out to be critical to derive these algorithm-based conditions, which helps establish convergence of projection-type methods for those problems with more general structures than the common monotonicity or Minty condition.




\begin{remark}{\ }

\begin{itemize}
\item If $F$ is monotone, then Condition (Local Minty) and Condition (Local Minty+) hold trivially (therefore so do Conditions (GP) and (GP+)).
\item If $\sol_m(\VI(F;\XX))\not=\emptyset$, then Condition (Local Minty) and Condition (Local Minty+) hold trivially (therefore so do Condition (GP) and (GP+)).
\item It is possible that $F$ is monotone but Conditions (Local Minty*) and (GP*) do not hold. Conversely, it is also possible that Conditions (Local Minty*) and (GP*) hold but $F$ is not monotone or $\weaksol=\emptyset$.
\end{itemize}
\end{remark}

The next two examples demonstrate problem instances where $\sol_m(\VI(F))=\emptyset$ while Condition (GP)/(GP+)/(GP*) still hold.

\begin{example} \label{F=-x}
Consider $\XX=[-1,1]$, $F(x)=-x$. In that case, $\sol(\VI(F;\XX))=\{-1,0,+1\}$ and $\sol_m(\VI(F;\XX))=\emptyset$. Note that
\[
M(x;t) = \left\{
\begin{array}{ll}
\min\{1,(1+t)x\}, & \mbox{ if $x>0$ }; \\
0, & \mbox{ if $x=0$ }; \\
\max\{-1,(1+t)x\}, & \mbox{ if $x<0$ }.
\end{array}
\right.
\]
In particular, for $x>0$ we choose $x^*=1$; for $x<0$ we choose $x^*=-1$; for $x=0$ we choose $x^*=0$. It is now easy to verify that Condition (GP)/(GP+)/(GP*) hold in this case. In fact, Condition (Local Minty)/(Local Minty+)/(Local Minty*) all hold in this example.
\end{example}

\begin{example}

Consider $\XX=\|x\|\le 1$, $F(x)=Qx$,  where $Q=\begin{pmatrix}-1&0\\0&1\end{pmatrix}$. In this case, $\strongsol=\left\{(1,0)^\top,(0,0)^\top,(-1,0)^\top\right\}$. None of them is a Minty solution, but we can show that $\VI(F;\XX)$ satisfies (Local Minty)/(Local Minty+)/(Local Minty+) for any fixed $t\in(0,1]$. For the first two conditions, it suffices to provide the following two observations as the proof. Also note that we can focus on $x_1\ge0$, since the behavior is symmetric for the case $x_1\le0$.
\begin{enumerate}
    \item For any $x=(x_1,x_2)^\top\in\XX$ such that $x_1\ge0$, the whole sequence generated by $M(x;t)$ and $M^+(x;t)$ will remain $x_1\ge0$. This can be easily verified.
    \item For any $x=(x_1,x_2)^\top\in\XX$ such that $x_1\ge0$, we have
    \[
    \langle F(x),x-x^*\rangle\ge0,
    \]
    for $x^*=(1,0)^\top$. Since the above inequality results in $x_2^2\ge x_1(x_1-1)$, which always holds for any $0\le x_1\le 1$.
\end{enumerate}
It remains to show that Condition (Local Minty*) also holds. Similarly, let us focus on $x_1\ge0$ and use $x^*=(1,0)^\top$. Let us denote $x^+=x-tF(x)=((1+t)x_1,(1-t)x_2)^\top$, then we have
\[
M(x;t) = \left\{
\begin{array}{ll}
x^+, & \mbox{ if $\|x^+\|\le 1$ }; \\
x^+/\|x^+\|, & \mbox{ if $\|x^+\|>1$ }.
\end{array}
\right.
\]
For the case $\|x^+\|\le 1$, the condition
\begin{equation}\label{ineq-LM*}
\langle F(x),M(x;t)-x^*\rangle\ge0
\end{equation}
reduces to
\[
(1-t)x_2^2\ge(1+t)x_1^2-x_1,
\]
where the RHS is always non-positive for $x_1\le\frac{1}{1+t}$, which is the case for $\|x^+\|\le1$. Therefore inequality~\eqref{ineq-LM*} holds. For $\|x^+\|>1$, condition~\eqref{ineq-LM*} can be reduced to
\[
(1-t)x_2^2\ge(1+t)x_1^2-x_1\cdot\|x^+\|.
\]
Since $\|x^+\|\ge (1+t)x_1$, the RHS is always non-positive and condition~\eqref{ineq-LM*} holds.

\end{example}

The next example shows that, even if $F$ is monotone, Condition (GP*) does not necessarily hold. Otherwise, since Condition (GP*) is sufficient for the gradient projection method to converge (as will be shown in the next section), 
monotonicity would have been sufficient for the convergence as well (which is not the case for the gradient projection method).

\begin{example}
    Consider $\XX=\|x\|\le 1$, $F(x)=Qx$,  where $Q=\begin{pmatrix}0&1\\-1&0\end{pmatrix}$. In this case, $\mbox{Sol(VI($F;\XX$))}=\left\{(0,0)^\top\right\}$. This problem is originated from the saddle point problem $\min\max_{\|x\|^2+\|y\|^2\le 1}xy$ and is monotone. For a small $\epsilon>0$, consider $x=(\epsilon,0)^\top$, where $F(x)=(0,-\epsilon)^\top$. 

    Consider first $t\le \frac{\sqrt{1-\epsilon^2}}{\epsilon}$. In this case, $M(x;t)=x-tF(x)=\epsilon\cdot(1,t)^\top$. Therefore, 
    \begin{eqnarray}
    2 (1+\delta) t \langle F(x),M(x;t)-x^*\rangle + \| M(x;t)-x\|^2=-2(1+\delta)t^2\epsilon^2+t^2\epsilon^2 <0.\nonumber
    \end{eqnarray}
    On the other hand, if $t>\frac{\sqrt{1-\epsilon^2}}{\epsilon}$, then $M(x;t)=(1,t)^\top\cdot(1+t^2)^{-\frac{1}{2}}$, and
    \[
    \|M(x;t)-x\|^2=\left\|\frac{(1-\epsilon\sqrt{1+t^2},t)^\top}{\sqrt{1+t^2}}\right\|^2=1+\epsilon^2-\frac{2\epsilon}{\sqrt{1+t^2}}\le 1+\epsilon^2,
    \]
    whereas
    \begin{eqnarray}
        2 (1+\delta) t \langle F(x),M(x;t)-x^*\rangle=-2(1+\delta)\epsilon\cdot t^2\cdot (1+t^2)^{-\frac{1}{2}}<-2(1+\delta)\cdot(1-\epsilon^2),\nonumber
    \end{eqnarray}
    where the last inequality we take $t=\frac{\sqrt{1-\epsilon^2}}{\epsilon}$. It is then clear that for small enough $\epsilon$, Condition (GP*) will not hold, even if $F$ is monotone.
\end{example}

\subsection{Convergence of projection-type methods}
In the previous section, we present several algorithm-based conditions that are defined for sequences generated from either the gradient projection mapping or the extra-gradient projection mapping. In this section we show how these conditions are applied in the convergence analysis for two projection-type methods, the vanilla gradient projection method and the extra-gradient method.

\subsubsection{The gradient projection method}
Consider the gradient projection method:
\begin{eqnarray}
    x^{k+1}&:=&\arg\min\limits_{x\in\XX}\langle F(x^k),x-x^k\rangle+\frac{1}{2t}\|x-x^k\|^2,\nonumber
\end{eqnarray}
or equivalently written as
\begin{eqnarray}
    x^{k+1}&:=& M(x^k;t).\nonumber
\end{eqnarray}

It is now clear why the conditions (Local Minty)/(Local Minty*) and their relaxations (GP)/(GP*) are defined for sequences generated from the gradient projection mapping. They assume that for each sequence generated by the gradient projection method there exists at least one solution $x^*\in\strongsol$ such that their respective defining inequalities continue to hold. We first derive a key intermediate inequality from the gradient projection update itself, which makes it clearer the exact condition to be used in the following convergence analysis.

\begin{lemma}\label{lem:gp}
    For the gradient projection method, we have
    \begin{equation}\label{iter-gp}
    \frac{1}{2} \|x^k-x^*\|^2 \ge \frac{1}{2} \|x^{k+1}-x^*\|^2 + t \langle F(x^k),x^{k+1}-x^*\rangle + \frac{1}{2} \| x^{k+1}-x^k\|^2 .
    \end{equation}
    for any $x^*\in\sol(\VI(F;\XX))$.
    \begin{proof}
    Since
    \begin{eqnarray}
        \langle tF(x^k)+x^{k+1}-x^k,x-x^{k+1}\rangle\ge0\quad\forall x\in\XX,\nonumber
    \end{eqnarray}
    with $x=x^*$, we have
    \begin{eqnarray}
    \langle tF(x^k),x^*-x^{k+1}\rangle\ge \frac{1}{2}\left(\|x^{k+1}-x^k\|^2+\|x^{k+1}-x^*\|^2-\|x^k-x^*\|^2\right).\nonumber
    \end{eqnarray}
    Rearranging terms gives the result.
    \end{proof}
\end{lemma}

In view of the inequality~\eqref{iter-gp}, it is straightforward that Conditions (Local Minty*) and (GP*) can provide a bound on $\langle F(x^k),x^{k+1}-x^*\rangle$ for the whole sequence with respect to some $x^*\in\strongsol$, thus the convergence follows. The results are summarized in the next theorem.

\begin{theorem}\label{thm:gp}
Under Condition (GP*), and assume $F$ is Lipschitz continuous with constant $L$, the gradient projection algorithm is convergent for $\VI(F;\XX)$. Moreover, 
\[
\min_{1\le k \le N} P(x^k) = O(1/N),\quad \min_{1\le k \le N} G(x^k) = O(1/N^{\frac{1}{2}}).
\]
\begin{proof}
    In view of Lemma~\ref{lem:gp} and Condition (GP*), we have:
    \begin{eqnarray}
        \frac{1}{2}\left(\|x^k-x^{*}\|^2-\|x^{k+1}-x^*\|^2\right)&\ge& t \langle F(x^k),x^{k+1}-x^*\rangle + \frac{1}{2} \| x^{k+1}-x^k\|^2\nonumber\\
        &\overset{\tiny \mbox{(GP*)}}{\ge}& -\frac{1}{2(1+\delta)}\|x^{k+1}-x^k\|^2+\frac{1}{2}\|x^{k+1}-x^k\|^2\nonumber\\
        &=& \frac{\delta}{2(1+\delta)}\|x^{k+1}-x^k\|^2.\nonumber
    \end{eqnarray}
    By Condition (GP), there exists an $x^*$ such that the above inequality holds for $k=1,...,N$. Therefore, summing up the inequality for $k=1,...,N$, we have:
    \begin{eqnarray}
        \sum\limits_{k=1}^N\|x^{k+1}-x^k\|^2\le \left(1+\frac{1}{\delta}\right)\|x^1-x^*\|^2,\nonumber
    \end{eqnarray}
    which implies
    \begin{eqnarray}
        \min\limits_{1\le k\le N}P(x^k)\le \frac{1}{N}\sum\limits_{k=1}^N\|x^{k+1}-x^k\|^2\le \frac{1}{N}\cdot\left(1+\frac{1}{\delta}\right)\|x^1-x^*\|^2=\mathcal{O}(1/N).\nonumber
    \end{eqnarray}
    Moreover, we can transform the measurement in $\min\limits_{1\le k\le N}P(x^k)$ into $\min\limits_{1\le k\le N}G(x^k)$. Since
    \begin{eqnarray}
        \langle tF(x^k)+x^{k+1}-x^k,x-x^{k+1}\rangle\ge0\quad\forall x\in\XX,\nonumber
    \end{eqnarray}
    we have
    \begin{eqnarray}
        \langle F(x^k),x^{k+1}-x\rangle\le -\frac{1}{t}(x^{k+1}-x)^\top(x^{k+1}-x^k)\le \frac{D}{t}\|x^{k+1}-x^k\|.\nonumber
    \end{eqnarray}
    Therefore,
    \begin{eqnarray}
        \langle F(x^{k+1}),x^{k+1}-x\rangle&=& \langle F(x^k),x^{k+1}-x\rangle+\langle F(x^{k+1})-F(x^k),x^{k+1}-x\rangle\nonumber\\
        &\le& \frac{D}{t}\|x^{k+1}-x^k\|+\|F(x^{k+1})-F(x^k)\|\cdot\|x^{k+1}-x\|\nonumber\\
        &\le& D\left(\frac{1}{t}+L\right)\|x^{k+1}-x^k\|.\nonumber
    \end{eqnarray}
    Define $k_N:=\arg\min\limits_{1\le k\le N}P(x^k)$, then
    \begin{eqnarray}
        \langle F(x^{k_N+1}),x^{k_N+1}-x\rangle&\le& D\left(\frac{1}{t}+L\right)\|x^{k_N+1}-x^{k_N}\|\nonumber\\
        &\le& D\left(\frac{1}{t}+L\right)\frac{1}{N^{\frac{1}{2}}}\left(1+\frac{1}{\delta}\right)^{\frac{1}{2}}\|x^1-x^*\|\nonumber\\
        &\le& D^2\left(\frac{1}{t}+L\right)\frac{1}{N^{\frac{1}{2}}}\left(1+\frac{1}{\delta}\right)^{\frac{1}{2}},\nonumber
    \end{eqnarray}
    which holds for all $x\in\XX$. This implies
    \begin{eqnarray}
        \min\limits_{1\le k\le N}G(x^k)=\max\limits_{x\in\XX}\langle F(x^{k_N}),x^{k_N}-x\rangle\le D^2\left(\frac{1}{t}+L\right)\frac{1}{N^{\frac{1}{2}}}\left(1+\frac{1}{\delta}\right)^{\frac{1}{2}}=\mathcal{O}(1/N^{\frac{1}{2}}).\nonumber
    \end{eqnarray}
\end{proof}
\end{theorem}

It is well-known that if $F$ is strongly monotone, gradient projection method can be guaranteed to converge with iteration complexity $\mathcal{O}\left(\frac{L^2}{\mu^2}\log\frac{1}{\epsilon}\right)$, while $F$ being merely monotone is insufficient for the convergence. On the other hand, Theorem~\ref{thm:gp} provides a different sufficient condition (GP*) such that gradient projection method converges globally in terms of $\min_{1\le k \le N} G(x^k)=G(x^{k_N})$. As shown in the previous examples, there exist problems which are not monotone but Condition (GP*) is satisfied.

\subsection{The extra-gradient method}
Consider
\[
\left\{
\begin{array}{ccl}
x^{k+0.5} &:=& \mbox{\rm arg} \min\limits_{x \in \XX} \langle F(x^k), x - x^k \rangle + \frac{1}{2t} \| x - x^k \|^2 \\
x^{k+1}   &:=& \mbox{\rm arg} \min\limits_{x \in \XX} \langle F(x^{k+0.5}), x - x^k \rangle + \frac{1}{2t} \| x - x^k \|^2 ,
\end{array}
\right.
\]
or it can be equivalently written as
\[
\left\{
\begin{array}{ccl}
x^{k+0.5} &:=& M(x^k;t) \\
x^{k+1}   &:=& M^+(x^k;t) ,
\end{array}
\right.
\]

Note that the extra-gradient method can be viewed as a special case of the ARE update discussed in Section~\ref{sec:are-conv} with $p=1$. Below we introduce a key inequality derived from the update, which also appears with a similar form in the analysis for ARE (cf.\ \eqref{iter}). Both the notation and the parameter constraint are slightly adjusted in the following lemma, therefore a proof is provided for completeness.  

\begin{lemma}\label{lem:EXgp}
    For the extra-gradient method, assume that $F$ is Lipschitz continuous with constant $L$, and $t\le \frac{1}{\sqrt{2}L}$, the following inequality holds:
    \begin{equation} \label{iter-t}
    \langle F(x^{k+0.5}),x^{k+0.5}-x\rangle + \frac{1}{4t} \|x^{k+0.5}-x^k\|^2 \le \frac{1}{2t} \left[ \|x^k-x\|^2-\|x^{k+1}-x\|^2 \right].
    \end{equation}
    \begin{proof}
        Optimality of $x^{k+0.5}$:
\begin{equation} \label{subVI2}
\langle F(x^k) + \frac{1}{t} (x^{k+\half}-x^k) , x - x^{k+\half}\rangle \ge 0,\,\, \forall x \in \XX.
\end{equation}
Substituting $x=x^{k+1}$ in \eqref{subVI2} we have
\begin{eqnarray}
&&\langle F(x^k)  ,  x^{k+1} - x^{k+\half}  \rangle \nonumber\\
&\ge& \frac{1}{t} \langle x^{k+\half} - x^k , x^{k+\half} - x^{k+1}  \rangle \nonumber \\
&=& \frac{1}{2t} \left( \|x^{k+\half} - x^k\|^2 + \|x^{k+1}-x^{k+\half}\|^2 - \| x^{k+1}-x^k\|^2 \right) . \label{tildeF-2}
\end{eqnarray}

On the other hand, by the optimality condition at $x^{k+1}$ we have
\[
\langle F(x^{k+\half})  + \frac{1}{t} (x^{k+1} - x^k) , x - x^{k+1} \rangle \ge 0 ,\mbox{ for all $x\in \XX$}.
\]
Hence,
\begin{eqnarray}
\langle F(x^{k+\half})  , x - x^{k+1} \rangle &\ge& \frac{1}{t} \langle x^{k+1} - x^k, x^{k+1} - x \rangle  \nonumber \\
&=& \frac{1}{2t} \left(  \| x^{k+1} - x\|^2 + \| x^{k+1} - x^k\|^2 - \| x^k - x\|^2 \right) ,\, \mbox{ for all $x\in \XX$}. \label{F ineq-2}
\end{eqnarray}

Continue with the above inequality, for any given $x\in \XX$ we have
\begin{eqnarray*}
& & \frac{1}{2t} \left(  \| x^{k+1} - x\|^2 + \| x^{k+1} - x^k\|^2 - \| x^k - x\|^2 \right) \\
&\overset{\eqref{F ineq-2}}{\le} & \langle F(x^{k+\half})  , x - x^{k+1} \rangle \\
&=&  \langle F(x^{k+\half})  , x - x^{k+\half} \rangle + \langle F(x^{k+\half})  , x^{k+\half} - x^{k+1} \rangle \\
&=& \langle F(x^{k+\half})  , x - x^{k+\half} \rangle + \langle F(x^{k+\half}) - F(x^k) , x^{k+\half} - x^{k+1} \rangle + \langle F(x^k)  , x^{k+\half} - x^{k+1} \rangle \\
&\le& \langle F(x^{k+\half})  , x - x^{k+\half} \rangle + \| F(x^{k+\half}) - F(x^k) \| \cdot \| x^{k+\half} - x^{k+1} \|
+ \langle F(x^k)  , x^{k+\half} - x^{k+1} \rangle \\
&\le & \langle F(x^{k+\half})  , x - x^{k+\half} \rangle + \frac{t\| F(x^{k+\half}) - F(x^k)\|^2}{2}  + \frac{ \|  x^{k+\half} - x^{k+1} \|^2}{2t}\\
& & + \langle F(x^k)  , x^{k+\half} - x^{k+1} \rangle \\
&\overset{}{\le} & \langle F(x^{k+\half})  , x - x^{k+\half} \rangle + \frac{ tL^2 \| x^{k+\half} - x^k\|^{2} }{2}  + \frac{ \|  x^{k+\half} - x^{k+1} \|^2}{2t}\\
& & + \langle F(x^k)  , x^{k+\half} - x^{k+1} \rangle .
\end{eqnarray*}
Since $t\le\frac{1}{\sqrt{2}L}$, and with \eqref{tildeF-2} we have
\begin{eqnarray*}
& & \frac{1}{2t} \left(  \| x^{k+1} - x\|^2 + \| x^{k+1} - x^k\|^2 - \| x^k - x\|^2 \right) \\
&\le& \langle F(x^{k+\half})  , x - x^{k+\half} \rangle + \frac{  \| x^{k+\half} - x^{k} \|^2}{4t} + \frac{ \|  x^{k+\half} - x^{k+1} \|^2}{2t} \\
&   & + \frac{1}{2t} \left[ - \|x^{k+\half} - x^k\|^2 - \|x^{k+1}-x^{k+\half}\|^2 + \| x^{k+1}-x^k\|^2 \right] .
\end{eqnarray*}
Canceling out terms, we simplify the above inequality into
\begin{equation} \nonumber
\langle F(x^{k+\half})  , x^{k+\half} - x \rangle + \frac{1}{4t} \| x^{k+\half} - x^{k} \|^2 \le
\frac{1}{2t} \left[ \|x^{k} - x\|^2 - \|x^{k+1}-x\|^2  \right] .
\end{equation}
    \end{proof}
\end{lemma}

One can immediately identify the connections between inequality~\eqref{iter-t} and Condition (GP+). In fact, inequality~\eqref{iter-t} plays the central role in the convergence of the extra-gradient method (or in general, the extra-gradient-type method such as ARE), and iteration complexities of different orders can be established following this inequality based on the conditions imposed on the VI model. The most conventional assumption will be the (strong) monotonicity of $F$, while in Section~\ref{sec:are-conv} it is relaxed to be Minty condition $\weaksol\neq\emptyset$. Here, Condition (GP+) provides a more direct way to guide the convergence analysis of the extra-gradient method based on inequality~\eqref{iter-t}, as summarized in the next theorem.

\begin{theorem}\label{thm:EXgp}
Under Condition (GP+), and assume $F$ is Lipschitz continuous with constant $L$, the extra-gradient method with $t\le \frac{1}{\sqrt{2}L}$ is convergent for $\VI(F;\XX)$. Moreover, 
\[
\min_{1\le k \le N} P(x^{k}) = O(1/N),\quad \min_{1\le k \le N} G(x^{k+0.5}) = O(1/N^{\frac{1}{2}}).
\]
\begin{proof}
    In view of Lemma~\ref{lem:EXgp} and Condition (GP+), we have:
    \begin{eqnarray}
         \frac{1}{2t} \left[ \|x^{k} - x^*\|^2 - \|x^{k+1}-x^*\|^2  \right]&\ge& \langle F(x^{k+\half})  , x^{k+\half} - x^* \rangle + \frac{1}{4t} \| x^{k+\half} - x^{k} \|^2\nonumber\\
         &\overset{\tiny \mbox{(GP+)}}{\ge}& -\frac{1}{4t(1+\delta)}\|x^{k+0.5}-x^k\|^2+\frac{1}{4t}\|x^{k+0.5}-x^k\|^2\nonumber\\
         &=& \frac{\delta}{1+\delta}\frac{1}{4t}\|x^{k+0.5}-x^k\|^2.\nonumber
    \end{eqnarray}
    Since Condition (GP+) also asserts that there exists an $x^*$ such that the above inequality holds for $k=1,...,N$, summing up the inequality for $k=1,...,N$ gives us:
    \begin{eqnarray}
        \sum\limits_{k=1}^N\|x^{k+0.5}-x^k\|^2\le 2\left(1+\frac{1}{\delta}\right)\|x^1-x^*\|^2,\nonumber
    \end{eqnarray}
    which implies
    \begin{eqnarray}
        \min\limits_{1\le k\le N}P(x^{k})\le \frac{1}{N}\sum\limits_{k=1}^N\|x^{k+0.5}-x^k\|^2\le \frac{2}{N}\cdot\left(1+\frac{1}{\delta}\right)\|x^1-x^*\|^2=\mathcal{O}(1/N).\nonumber
    \end{eqnarray}
    Moreover, we can transform the measurement in $\min\limits_{1\le k\le N}P(x^{k})$ into $\min\limits_{1\le k\le N}G(x^{k+0.5})$. Since
    \begin{eqnarray}
        \langle tF(x^k)+x^{k+0.5}-x^k,x-x^{k+0.5}\rangle\ge0\quad\forall x\in\XX,\nonumber
    \end{eqnarray}
    we have
    \begin{eqnarray}
        \langle F(x^k),x^{k+0.5}-x\rangle\le -\frac{1}{t}(x^{k+0.5}-x)^\top(x^{k+0.5}-x^k)\le \frac{D}{t}\|x^{k+0.5}-x^k\|.\nonumber
    \end{eqnarray}
    Therefore,
    \begin{eqnarray}
        \langle F(x^{k+0.5}),x^{k+0.5}-x\rangle&=& \langle F(x^k),x^{k+0.5}-x\rangle+\langle F(x^{k+0.5})-F(x^k),x^{k+0.5}-x\rangle\nonumber\\
        &\le& \frac{D}{t}\|x^{k+0.5}-x^k\|+\|F(x^{k+0.5})-F(x^k)\|\cdot\|x^{k+0.5}-x\|\nonumber\\
        &\le& D\left(\frac{1}{t}+L\right)\|x^{k+0.5}-x^k\|.\nonumber
    \end{eqnarray}
    Define $k_N:=\arg\min\limits_{1\le k\le N}P(x^{k})$, then
    \begin{eqnarray}
        \langle F(x^{k_N+0.5}),x^{k_N+0.5}-x\rangle&\le& D\left(\frac{1}{t}+L\right)\|x^{k_N+0.5}-x^{k_N}\|\nonumber\\
        &\le& D\left(\frac{1}{t}+L\right)\frac{\sqrt{2}}{N^{\frac{1}{2}}}\left(1+\frac{1}{\delta}\right)^{\frac{1}{2}}\|x^1-x^*\|\nonumber\\
        &\le& D^2\left(\frac{1}{t}+L\right)\frac{\sqrt{2}}{N^{\frac{1}{2}}}\left(1+\frac{1}{\delta}\right)^{\frac{1}{2}},\nonumber
    \end{eqnarray}
    which holds for all $x\in\XX$. This implies
    \begin{eqnarray}
        \min\limits_{1\le k\le N}G(x^{k+0.5})=\max\limits_{x\in\XX}\langle F(x^{k_N+0.5}),x^{k_N+0.5}-x\rangle\le D^2\left(\frac{1}{t}+L\right)\frac{\sqrt{2}}{N^{\frac{1}{2}}}\left(1+\frac{1}{\delta}\right)^{\frac{1}{2}}=\mathcal{O}(1/N^{\frac{1}{2}}).\nonumber
    \end{eqnarray}
\end{proof}
\end{theorem}

The convergence rate in Theorem~\ref{thm:EXgp} turns out to be the same as the rate in Theorem~\ref{thm:gp} for gradient projection method when Condition (GP*) is satisfied instead, as well as the rate in Theorem~\ref{thm:are-Minty} for ARE ($p=1$) when the Minty condition is satisfied. As discussed in the earlier examples, there exists problems where Conditions (GP+) or (GP*) are satisfied but no Minty solution exists. On the other hand, while these assumptions are able to provide alternative sufficient conditions for the convergence of certain class of algorithms, in general they can be difficult to verify {\it a priori} due to the requirement to hold for the whole sequence.


\section{Conclusion}
\label{sec:conclusion}

In this paper, we discuss sufficient conditions for projection-type methods to converge in VI problems that are not necessarily monotone. We first focus on the problem where a Minty solution exists, which is a relaxation of the monotonicity assumption. We derive the guaranteed global convergence rate for a general extra-gradient type method ARE under the Minty condition, and then we extend the discussion to properties and implications of Minty solutions in more specific problem classes such as optimization and Nash games. Finally, we present conditions on VI problems that are algorithm-based, in the sense that they are closely connected to the algorithms we are interested in (in particular, projection-type methods) and can suitably serve as sufficient conditions to guarantee the convergence. Conventionally, the algorithms are designed for problems where assumptions on the structure are made {\it a priori}, and the convergence is only guaranteed under these assumptions. In this paper, we provide an alternative aspect, by ``desinging'' conditions on the VI model such that they are sufficient to guarantee convergence of certain class of algorithms. We show that this approach is indeed capable of characterizing different classes of VI problems (potentially broader) from the existing ones such as monotone VI or VI with Minty solutions. We analyze the convergence of gradient projection method and extra-gradient method under the proposed conditions. There are still questions remaining, such as: if there are other algorithm-based conditions that can be derived for different projection-type methods or non-projection-type methods; if there exist different characterizations of these conditions such that they can be more easily verified. 
Answering these questions require some efforts in the future research.

\printbibliography

\end{document}